\documentclass[letterpaper,11pt]{amsart}

\usepackage[margin=1.4in]{geometry}
\usepackage{amscd, amssymb}
\usepackage{amsmath,amscd}
\usepackage{mathtools}
\usepackage[driverfallback=hypertex]{hyperref}
\usepackage{graphics}
\usepackage{graphicx}
\usepackage{color}
\usepackage{bm}
\usepackage{enumitem}
\usepackage{chessboard,skak}
\setlist[enumerate,1]{label={(\alph*)}}
\setlist[enumerate,2]{label={(\roman*)}}

\newif\ifdraft
\ifdraft
\usepackage{draftwatermark}
\SetWatermarkText{DRAFT}
\SetWatermarkScale{2}
\fi
\newtheorem{thm}{Theorem}[section]
\newtheorem{prop}[thm]{Proposition}

\newtheorem{lem}[thm]{Lemma}

\newtheorem{clm}[thm]{Claim}

\theoremstyle{definition}
\newtheorem{definition}[thm]{Definition}

\newtheorem*{notation*}{Notation}

\theoremstyle{remark}

\newtheorem{obs}[thm]{Observation}
\newtheorem{algorithm}[thm]{Algorithm}

\newcommand{\ignore}[1]{}
\newcommand{\R}{\mathbb R}

\newcommand{\N}{\mathbb N}
\newcommand{\Prob}{{\mathbb{P}}}

\newcommand{\mB}{{\mathcal{B}}}

\newcommand{\mA}{{\mathcal{A}}}
\newcommand{\mF}{{\mathcal{F}}}

\newcommand{\mL}{{\mathcal{L}}}

\newcommand{\mQ}{{\mathcal{Q}}}

\newcommand{\mS}{{\mathcal{S}}}

\newcommand{\E}{{\mathbb{E}}}

\newcommand{\tO}{{\tilde{O}}}

\newcommand{\tQ}{{\tilde{Q}}}
\newcommand{\tR}{{\tilde{R}}}

\newcommand{\elltor}[1]{{\ell_{#1}^{\text{tor}}}}
\newcommand{\ellfar}[1]{{\ell_{#1}^{\text{far}}}}

\newcommand{\oone}{{o \left(1\right)}}

\newcommand{\termdefine}[1]{\textbf{#1}}

\newcommand{\given}{{|}}

\begin{document}
\title{A lower bound for the $n$-queens problem}

\author{Zur Luria}
\address{Software Department, Azrieli College of Engineering, Jerusalem, Israel}
\email{zluria@gmail.com}

\author{Michael Simkin}
\address{Harvard University Center of Mathematical Sciences and Applications, Cambridge, MA, USA.}
\email{msimkin@cmsa.fas.harvard.edu}

\begin{abstract}
	The $n$-queens puzzle is to place $n$ mutually non-attacking queens on an $n \times n$ chessboard. We present a simple two stage randomized algorithm to construct such configurations. In the first stage, a random greedy algorithm constructs an approximate \textit{toroidal} $n$-queens configuration. In this well-known variant the diagonals wrap around the board from left to right and from top to bottom. We show that with high probability this algorithm succeeds in placing $(1-o(1))n$ queens on the board. In the second stage, the method of absorbers is used to obtain a complete solution to the non-toroidal problem. By counting the number of choices available at each step of the random greedy algorithm we conclude that there are more than $\left( \left( 1 - o(1) \right) n e^{-3} \right)^n$ solutions to the $n$-queens problem. This proves a conjecture of Rivin, Vardi, and Zimmerman in a strong form.
\end{abstract}

\maketitle

\pagestyle{plain}

\section{Introduction}

%TODO: Perhaps discuss in introduction: Other algorithms for constructing configurations, and the use of this problem as a test bed for algorithm development.

An $n$-queens configuration is a placement of $n$ mutually non-attacking queens on an $n \times n$ chessboard. Equivalently, it is an order-$n$ permutation matrix in which the sum of each diagonal is at most $1$.

Let $\mQ(n)$ be the number of solutions to the $n$-queens puzzle. The problem of determining $\mQ(n)$ was posed in 1848 by German chess enthusiast Max Bezzel (specifically for $n=8$). It soon attracted the attention of several eminent mathematicians, including Gauss. P\'{o}lya also considered the toroidal $n$-queens problem, in which the diagonals wrap around the board from left to right and from top to bottom. We call such diagonals \termdefine{toroidal diagonals}. Although our main result concerns the classical $n$-queens problem, the toroidal problem features prominently in our proof.
For more on the history of the problem, as well as a survey of results and a list of open problems, we refer the reader to Bell and Stevens \cite{bell2009survey}.

To date, the best known upper bound on $\mQ(n)$ is due to the first author \cite[Theorem 1.4]{luria2017new}: For $\alpha = 1.587$, it holds that $\mQ(n) \leq \left( \left( 1 + \oone \right) ne^{-\alpha} \right)^n$.
Regarding lower bounds, it is not immediately obvious that $\mQ(n)>0$ for all (large enough) $n$. Pauls \cite{pauls1874maximalproblem,pauls1874maximalproblemII} showed that this holds for all $n \geq 4$. Thus, $\mQ(n) = \Omega (1)$, and to our knowledge this is the best known general lower bound. Rivin, Vardi, and Zimmerman \cite[Theorem 1]{rivin1994n} gave an exponential lower bound on $\mQ(n)$ that holds for a dense set of integers, and the first author \cite[Theorem 1.1]{luria2017new} showed that for infinitely many $n$, the superexponential bound $\mQ(n) \geq n^{n/16-o(n)}$ holds. 

Our main result is the following lower bound on $\mQ(n)$.

\begin{thm}\label{thm:main}
	For all natural numbers $n$,
	\[
	\mQ(n) \geq \left( \left( 1 - \oone \right) \frac{n}{e^3} \right)^n.
	\]
\end{thm}

As a consequence of Theorem \ref{thm:main} and the trivial bound $\mQ(n) \leq n!$, we conclude that
\[
\lim_{n\to\infty} \frac{\log \mQ(n)}{n \log n} = 1.
\]

This proves, in a strong form, a conjecture of Rivin, Vardi, and Zimmerman \cite[Conjecture 1]{rivin1994n} who conjectured that the limit above is positive.

It is interesting to note that the previous lower bounds on $\mQ(n)$ are actually due to algebraic constructions of solutions to the toroidal $n$-queens puzzle. In contrast, our construction relies purely on probabilistic and combinatorial methods: random greedy algorithms and absorption. The former has proved to be a powerful tool when constructing constrained combinatorial objects such as regular graphs \cite{rucinski1992random}, triangle Ramsey graphs \cite{bohman2009triangle}, and high-girth approximate Steiner triple systems \cite{glock2020conjecture,bohman2019large}. Combined with absorbers, they also played important roles in the recent constructions of combinatorial designs \cite{keevash2014existence,glock2016existence}. In the latter example, a random greedy algorithm was the key to obtaining lower bounds on the count of the associated objects \cite{keevash2018counting}.

We prove Theorem \ref{thm:main} by describing a probabilistic algorithm to construct solutions to the $n$-queens puzzle. We formally describe the algorithm in Section \ref{sec:full algorithm}, and we give a brief overview here. The algorithm has two phases. The first is a random greedy algorithm that constructs an approximate toroidal $n$-queens configuration. Beginning with an empty board, we repeatedly place queens uniformly at random subject to the constraint that they may not share a row, column, or toroidal diagonal. We show that with high probability this algorithm succeeds in placing ${(1-o(1))n}$ queens on the board. This phase of the algorithm, which is analyzed in Section \ref{sec:rga}, also gives the enumeration used to prove Theorem \ref{thm:main}. In the second phase of the algorithm we use absorbers to perform a small number of modifications and obtain a complete solution. The absorbing phase is analyzed in Section \ref{sec:absorbers}.

In the random greedy phase of the algorithm it would be natural to choose each queen uniformly at random subject to the constraint that it may not share a row, column, or (non-toroidal) diagonal with a previously placed queen. However, we have found this algorithm challenging to analyze. The difficulty is due to the fact that the lengths of the diagonals vary, leading to a non-uniform distribution of the queens. In contrast, toroidal diagonals all have the same length. This simplifies the analysis of the random greedy algorithm.

\subsection{Notation and terminology}

We adopt the standard notation $[n]:=\{1, ... ,n\}$. Also, we write $a = b \pm c$ as a shorthand for $b-|c| \leq a \leq b+|c|$.

Throughout, all asymptotics are with $n\to\infty$. Since we are concerned only with asymptotic statements we assume that $n$ is sufficiently large for relevant inequalities to hold. For example, we may write $n^2-n > n$ without explicitly requiring $n > 2$.

We say that a sequence of events occurs \termdefine{with high probability} (\termdefine{w.h.p.}) if the probabilities of their occurrence tend to $1$. Typically, we will abuse terminology and say that a \textit{specific} event occurs w.h.p.\ while leaving the sequential statement to be inferred from context.

To formally define the classical and toroidal $n$ queens problems, we identify the $n \times n$ chessboard with $[n]^2$. In the classical problem, a queen at $(x,y)$ attacks all squares $(x',y')$ such that either $x'=x$, $y'=y$, $x'+y'=x+y$, or $x'-y'=x-y$. In the toroidal problem, these equations are replaced by equalities modulo $n$. The toroidal diagonals containing the square $(x,y)$ are
$\{(x',y')\in [n]^2: x'+y' = x+y \bmod n\}$ and $\{(x',y')\in [n]^2: x'-y' = x-y \bmod n\}$. Note that there are $n$ toroidal diagonals of each type, all of size $n$. A \termdefine{partial $n$-queens configuration} is a set $Q \subset [n]^2$ such that every row, column, and diagonal contains at most one element of $Q$. An $n$-queens configuration is a partial $n$-queens configuration of size $n$. Toroidal configurations are defined similarly using toroidal diagonals.

\section{A randomized construction of $n$-queens configurations}\label{sec:full algorithm}

In this section we formally define the algorithm we use to construct $n$-queens configurations. As mentioned in the introduction, we begin by using a random greedy algorithm to construct a partial \textit{toroidal} $n$-queens configuration.

Fix (a sufficiently large) $n \in \N$, and set
\[
\alpha = 10^{-4}, \quad T = \lfloor(1 - n^{-\alpha}) n \rfloor.
\]

\begin{algorithm}\label{alg:rga}
	\hfill
	\begin{itemize}
		\item Set $Q(0) = \emptyset$.
		  
		 % For this to be a well-defined algorithm, we must define T here.
		  % Alternatively, T could be an input to the algorithm, in which case we would prove that for some T=(1-o(1))n, the algorithm doesn't abort.
		
		\item For each $0 \leq t < T$, let $\mA(t) \subseteq [n]^2$ be the set of \termdefine{available} positions, i.e., those positions lying in a row, column, and toroidal diagonals unoccupied by $Q(t)$.
		
		\item If $\mA(t) = \emptyset$, abort.
		
		\item Otherwise, choose $(x_t,y_t) \in \mA(t)$ uniformly at random and independently of previous choices. Set $Q(t+1) = Q(t) \cup \{(x_t,y_t)\}$.
	\end{itemize}
\end{algorithm}
For technical reasons, if the algorithm aborts at time $t$, then for every $t < s \leq T$ we set $Q(s) = Q(t)$.

In Proposition \ref{prop:dynamic concentration} we show that w.h.p.\ Algorithm \ref{alg:rga} does not abort. In the next phase of the algorithm we hope to remove a small number of queens from the board, and then complete the puzzle. The key is the idea of \textit{absorption}, which we now illustrate: Suppose $Q \subseteq [n]^2$ is a partial $n$-queens configuration that does not cover row $r$ and column $c$. We wish to obtain a partial $n$-queens configuration $Q'$ that covers all rows and columns covered by $Q$, and also covers row $r$ and column $c$. We might try adding $(r,c)$ to $Q$, but if either of the diagonals running through $(r,c)$ is occupied this will not work. Instead, we will look for a queen $(x,y) \in Q$ that satisfies:
\begin{enumerate}
	\item\label{itm:distinct diags} $(r,y)$ and $(x,c)$ do not share a diagonal (equivalently, $(r,c)$ and $(x,y)$ do not share a diagonal),
	
	\item\label{itm:free diags} none of the diagonals containing $(r,y)$ or $(x,c)$ are occupied.
\end{enumerate}
Supposing such a queen exists, we observe that $Q' \coloneqq \left( Q \setminus \{(x,y)\} \right) \cup \{(r,y),(x,c)\}$ is a partial $n$-queens configuration satisfying the conditions above. In this way, we have absorbed row $r$ and column $c$ into our configuration. We call such a queen an \termdefine{absorber for $(r,c)$ in $Q$}. We denote the set of absorbers for $(r,c)$ in $Q$ by $\mB_Q(r,c)$.

The following algorithm attempts to use absorbers to complete $Q(T)$ (the outcome of Algorithm \ref{alg:rga}).

\begin{algorithm}\label{alg:absorption}\hfill
	\begin{itemize}
		\item Let $L_R$ and $L_C$ be, respectively, the sets of rows and columns not covered by $Q(T)$. Note that $|L_R| = |L_C| \eqqcolon k$.
		
		\item Let $M = (r_1,c_1),(r_2,c_2),\ldots,(r_{k}, c_{k})$ be an arbitrary matching of $L_R$ and $L_C$.
		
		\item Set $R(0) \coloneqq Q(T)$.
		
		\item For $i=1,2,\ldots,k$:
		
		\begin{itemize}
			\item If $\mB_i \coloneqq \mB_{R(i-1)}(r_i,c_i) = \emptyset$ abort. 
			\item Otherwise, choose some $(x_i,y_i) \in \mB_i$ and set $R(i) \coloneqq \left( Q(i-1) \setminus \{(x_i,y_i)\} \right) \cup \{(x_i,c_i),(r_i,y_i)\}$.
		\end{itemize}
	\end{itemize}
\end{algorithm}

Clearly, if Algorithm \ref{alg:absorption} does not abort then $R(k)$ is an $n$-queens configuration. In Section \ref{sec:absorbers} we show that w.h.p.\ $Q(T)$ satisfies a combinatorial condition that guarantees the success of Algorithm \ref{alg:absorption}.

\section{Random queens on a torus}\label{sec:rga}

In this section we analyze Algorithm \ref{alg:rga}.

For $0 \leq t \leq T$ define the functions
\[
p(t) = 1 - \frac{t}{n}, \quad \varepsilon(t) = \frac{n^{0.51}}{ p(t)^{50}}.
\]

We refer to rows, columns, and toroidal diagonals collectively as \termdefine{lines}. Let $\mL$ be the set of lines. We note that $|\mL| = 4n$ and that each line has length $n$. For each $\ell \in \mL$ and $0 \leq t \leq T$, let $S_\ell(t)$ denote the number of available elements in line $\ell$ at time $t$. We say that a line is \termdefine{occupied at time $t$} if $Q(t)$ contains a position in the line.

\begin{prop}\label{prop:dynamic concentration}
	The following holds with probability $1 - \exp\left( - \Omega \left(n^{0.01}\right) \right)$. For every $\ell \in \mL$ and $0 \leq t \leq T$, if line $\ell$ is unoccupied at time $t$ then
	\[
	S_\ell(t) = n p(t)^3 \pm \varepsilon(t).
	\]
	Additionally, with probability $1 - \exp\left( - \Omega \left(n^{0.01}\right) \right)$, for every $0 \leq t \leq T$, it holds that $|\mA(t)| = n^2p(t)^4 \pm np(t) \varepsilon(t)$.
\end{prop}

A particular consequence of this proposition is that with probability $1 - \exp\left( - \Omega \left(n^{0.01}\right) \right)$ the algorithm does not abort. This is because for large enough $n$, for every $0 \leq t \leq T$ there holds $n^2p(t)^4 - np(t)\varepsilon(t) > 0$.

Proposition \ref{prop:dynamic concentration} can be interpreted as justifying the following independence heuristic: For every $t$ and every line $\ell \in \mL$, the probability that line $\ell$ is unoccupied at time $t$ is $p(t)$. If we assume independence then for every position the probability that it is available at time $t$ is $p(t)^4$. Hence, the expected number of available positions is $n^2p(t)^4$. If we condition on a specific line being unoccupied, then the expected number of available positions in it is $np(t)^3$.

We will use the following version of the Azuma-Hoeffding inequality.

\begin{thm}[{\cite[Lemma 1]{wormald1995differential}}]\label{thm:azuma}
	Let $X_0,X_1,\ldots$ be a supermartingale with respect to a filtration $\mF_0,\mF_1,\ldots$. For every $i \in \N$, let $a_i$ satisfy $a_i \geq |X_i - X_{i-1}|$. Then, for every $\lambda>0$ and $t \geq 0$, it holds that
	\[
	\Prob \left[ X_t \geq X_0 + \lambda \right] \leq \exp \left( - \frac{\lambda^2}{2 \sum_{i=1}^t a_i^2} \right).
	\]
\end{thm}

\begin{proof}[Proof of Proposition \ref{prop:dynamic concentration}]
	For convenience, we define the function
	\[
	s(t) = n p(t)^3.
	\]
	
	Let the stopping time $\tau$ be the smallest time $t \leq T$ at which there exists some $\ell \in \mL$ that is unoccupied and $\left| S_\ell(t) - s(t) \right| > \varepsilon (t)$. Set $\tau = T+1$ if there is no such time. It suffices to show that $\tau = T+1$ with sufficiently high probability. 
	
	Note that for  $t < \tau$, there are $n-t = np(t)$ unoccupied rows. In each, there are $s(t) \pm \varepsilon(t) = np(t)^3 \pm \varepsilon(t)$ available positions (where we have used the fact that $\tau > t$). Thus, 
    \begin{equation}\label{eq:total available}
    |\mA(t)| = np(t) \left( np(t)^3 \pm \varepsilon(t) \right) = n^2p(t)^4 \pm n p(t) \varepsilon(t).
    \end{equation} 
	
	For every $\ell \in \mL$, we also define the time $\tau_\ell$ as the minimum between $\tau$ and the last time at which line $\ell$ is unoccupied. For every $\ell \in \mL$, we define the shifted random variables
	\[
	S_\ell^+(t) = 
	\begin{cases}
	S_\ell(t) - s(t) - \frac{1}{2}\varepsilon(t) & \tau_\ell \geq t\\
	S_\ell^+(t-1) & \text{otherwise}
	\end{cases}
	\]
	and
	\[
	S_\ell^-(t) = 
	\begin{cases}
	s(t) - S_\ell(t) - \frac{1}{2}\varepsilon(t) & \tau_\ell \geq t\\
	S_\ell^-(t-1) & \text{otherwise}
	\end{cases}.
	\]
	The motivation of freezing the random variables at time $\tau_\ell$ is to ease the calculation of the expected one-step change: On the one hand, if $t < \tau$, then $Q(t)$ has a regularity that facilitates this calculation. On the other, if $t \geq \tau$, then the one-step change is definitionally zero. Additionally, at the moment $\ell$ becomes occupied, $S_\ell$ jumps to $0$. Thus, we cannot expect $S_\ell$ itself to follow a smooth trajectory.
	
	We note that $\tau \leq T$ only if for some $\ell \in \mL$ and some $0 \leq t \leq T$ one of $S_\ell^+(t)$ or $S_\ell^-(t)$ is greater than $\varepsilon(t) / 2$. This is the event whose probability we will bound.
	
	We now show that $\{S_\ell^+(t)\}_{t=0}^T$ and $\{S_\ell^-(t)\}_{t=0}^T$ are supermartingales with respect to the filtration $\{Q(t)\}_{t=0}^T$ induced by the random greedy algorithm. We first observe that for every $t<T$ it holds that $| S_\ell^+(t+1) - S_\ell^+(t) |, | S_\ell^-(t+1) - S_\ell^-(t) | = O(1)$. Indeed, when placing a queen on the board, it occupies exactly four lines. Each of these intersects a given line in exactly one position. Thus, $|S_\ell(t+1)-S_\ell(t)| \leq 4$. Additionally, by their definitions, $s(t)$ and $\varepsilon(t)$ change by only $O(1)$ in each time step.
	
	Next, we calculate expected one-step changes. If $t \geq \tau_\ell$ then, by definition, the expected change is $0$. It remains to show that, conditioning on $\tau > t$ and line $\ell$ remaining unoccupied at time $t+1$, the expected one-step change is non-positive. Consider a position $(r,c)$ in line $\ell$ that is available at time $t$. It becomes unavailable at time $t+1$ if and only if one of the other three lines containing it is occupied at time $t$. Since $\tau > t$, each of these lines contains $s(t) \pm \varepsilon(t)$ available positions, and they intersect only at $(r,c)$. Let $B$ be the event that line $\ell$ is unoccupied at time $t+1$. Observe that conditioning on $\tau > t$ and $B$ implies that $(x_{t+1},y_{t+1})$ is chosen uniformly from among 
	\[
	|\mA(t)| - S_\ell(t) \stackrel{\text{\eqref{eq:total available}, $t<\tau$}}{=} n^2p(t)^4 - np(t)^3 \pm \left(np(t)\varepsilon(t) + \varepsilon(t) \right) = n^2p(t)^4 \pm 3np(t)\varepsilon(t)
	\]
	positions. Finally, since we condition on $B$, there are $3(s(t) \pm \varepsilon(t)) \pm O(1) = 3s(t) \pm 4\varepsilon(t)$ positions that, if chosen at time $t+1$, make $(r,c)$ unavailable. Thus:
	\[
	\Prob \left[ (r,c) \notin \mA(t+1) \given \tau > t, B \right] = \frac{3s(t) \pm 4\varepsilon(t)}{n^2p(t)^4 \pm 3np(t)\varepsilon(t)} = \frac{3}{np(t)} \left( 1 \pm \frac{5\varepsilon(t)}{np(t)^3} \right).
	\]
	Therefore
	\begin{align*}
	\E \left[ S_\ell(t+1) - S_\ell(t) \given \tau > t, B \right] & = - S_\ell(t) \frac{3}{np(t)} \left( 1 \pm \frac{5\varepsilon(t)}{np(t)^3} \right)\\
	& \stackrel{\tau > t}{=} - \left( s(t) \pm \varepsilon(t) \right) \frac{3}{np(t)} \left( 1 \pm \frac{5\varepsilon(t)}{np(t)^3} \right)\\
	& = - 3p(t)^2 \left( 1 \pm \frac{7\varepsilon(t)}{np(t)^3} \right)
	= - 3 p(t)^2 \pm \frac{21 \varepsilon(t)}{np(t)}.
	\end{align*}
	We also observe that
	\[
	s(t+1) - s(t) = -\frac{3 n p(t)^2}{n} \pm \frac{6}{n} = - 3 p(t)^2 \pm \frac{\varepsilon(t)}{np(t)}
	\]
	and that
	\[
	\frac{1}{2} \left( \varepsilon(t+1) - \varepsilon(t) \right) = \frac{50 \varepsilon(t)}{2np(t)} \pm \frac{\varepsilon(t)}{np(t)}.
	\]
	Therefore:
	\[
	\E \left[ S_\ell^+(t+1) - S_\ell^+(t) \given \mF_t \right] = -3 p(t)^2 + 3 p(t)^2 - \frac{50 \varepsilon(t)}{2np(t)} \pm \frac{23 \varepsilon(t)}{np(t)} \leq 0
	\]
	and
	\[
	\E \left[ S_\ell^-(t+1) - S_\ell^-(t) \given \mF_t \right] = 3 p(t)^2 - 3 p(t)^2 - \frac{50 \varepsilon(t)}{2np(t)} \pm \frac{23 \varepsilon(t)}{np(t)} \leq 0.
	\]
	We conclude that $\{S_\ell^+(t)\}_{t=0}^T$ and $\{S_\ell^-(t)\}_{t=0}^T$ are supermartingales with maximal one-step change $O(1)$.
	
	We now apply Theorem \ref{thm:azuma} with $\lambda = \varepsilon(t) / 2$ to conclude that for every $0 \leq t \leq T$:
	\begin{align*}
	\Prob \left[ S_\ell^+(t) \geq S_\ell^+(0) + \varepsilon(t)/2  \right] & \leq \exp \left( - \Omega \left( \frac{n^{1.02}}{np(t)^{100}} \right)  \right)\\
	& \leq \exp \left( - \Omega \left( \frac{n^{1.02}}{np(T)^{100}} \right)  \right)
	= \exp \left( - \Omega \left( n^{0.01} \right) \right).
	\end{align*}
	Similarly:
	\[
	\Prob \left[ S_\ell^-(t) \geq S_\ell^-(0) + \varepsilon(t)/2 \right] \leq \exp \left( - \Omega \left( n^{0.01} \right) \right).
	\]
	Hence:
	\[
	\Prob \left[ \tau \leq T \right] \leq \exp \left( - \Omega \left( n^{0.01} \right) \right),
	\]
	completing the proof.
\end{proof}

\section{Absorbers}\label{sec:absorbers}

In this section we analyze Algorithm \ref{alg:absorption}. We wish to show that it is unlikely to abort. The next lemma provides a sufficient condition.

\begin{definition}\label{def:epsilon absorbing}
	Let $\ell > 0$. A partial $n$-queens configuration $Q$ is \termdefine{$\ell$-absorbing} if for every $(r,c) \in [n]^2$, it holds that $\left|\mB_Q(r,c)\right| \geq \ell$.
\end{definition}

\begin{lem}\label{lem:absorbing procedure works}
	Suppose $|Q(T)| = T$ and $Q(T)$ is $10 (n-T)$-absorbing. Then Algorithm \ref{alg:absorption} does not abort.
\end{lem}

\begin{proof}
	Let $k = n-T$. Algorithm \ref{alg:absorption} aborts only if for some $t$ we have $\left| \mB_{R(t)}(r_t,c_t) \right| = 0$. Thus, if for every $0 \leq t < k$ it holds that $\left| \mB_{R(t)}(r_t,c_t) \right| \geq 10k - 9t$, then the algorithm does not abort.
	
	We prove this inductively. Let $(r,c) \in [n]^2$. By assumption, $|\mB_{R(0)}(r,c)| \geq 10k$. Now assume that for $0 \leq t < k-1$ it holds that $|\mB_{R(t)}| \geq 10k - 9t$. We wish to show that $|\mB_{R(t+1)}| \geq 10k - 9t - 9$. We note that $R(t+1)$ is obtained from $R(t)$ by removing one queen and adding two queens. Thus, the conclusion follows from the following two observations regarding partial $n$-queens configurations $P_1, P_2 \subseteq [n]^2$:
	\begin{itemize}
		\item If $P_2$ is obtained from $P_1$ by removing a queen then $|\mB_{P_2}(r,c)| \geq |\mB_{P_1}(r,c)| - 1$. Indeed, if $(x,y) \in \mB_{P_1}(r,c)$ then it does not share a diagonal with $(r,c)$ and the four diagonals passing through $(r,y),(x,c)$ are unoccupied. Since $P_2 \subseteq P_1$, these facts hold in $P_2$ as well. Thus, if $(x,y) \in P_2$ then $(x,y) \in \mB_{P_2}(r,c)$. Since $|P_1 \setminus P_2| = 1$, the conclusion follows.
		
		\item If $P_2$ is obtained from $P_1$ by adding a queen then $|\mB_{P_2}(r,c)| \geq |\mB_{P_1}(r,c)| - 4$. Indeed, suppose $(x,y) \in \mB_{P_1}(r,c) \setminus \mB_{P_2}(r,c)$. Let $(q_r,q_c)$ be the additional queen that is present in $P_2$. By assumption, $(x,y)$ does not share a diagonal with $(r,c)$ and $(x,y) \in P_2$. Since $(x,y) \notin \mB_{P_2}(r,c)$, this means that one of the diagonals containing $(r,y)$ or $(x,c)$ is occupied by $(q_r,q_c)$. Therefore, one of the following holds:
		\[
		q_r+q_c = r+y,\quad q_r+q_c = x+c,\quad q_r-q_c = r-y, \quad q_r-q_c = x-c.
		\]
		Since $r,c,q_r,q_c$ are fixed, each equality has at most one solution in $[n]$. Now, for each $x$ that solves one of the equations, there is at most one $y$ such that $(x,y) \in P_1$ (since $P_1$ contains at most one queen in each column). Similarly, for each $y$ solving one of the equations there is at most one $x$ such that $(x,y) \in P_1$. Therefore there are at most $4$ absorbers $(x,y) \in \mB_{P_1}(r,c)$ for which one of the equalities holds.
	\end{itemize}
\end{proof}

The next lemma asserts that w.h.p.\ $Q(T)$ contains $\Omega(n)$ absorbers for every position. By Proposition \ref{prop:dynamic concentration}, w.h.p.\ $n - |Q(T)| = n - T = o(n)$. Thus, it then follows from Lemma \ref{lem:absorbing procedure works} that Algorithm \ref{alg:absorption} succeeds in constructing an $n$-queens configuration.

\begin{lem}\label{lem:an abundance of absorbers}
	There exists a constant $\beta > 0$ such that w.h.p.\ $Q(T)$ is $(\beta n)$-absorbing.
\end{lem}

The intuition is that $Q(T)$ contains approximately $n$ queens, each occupying a single diagonal of each type (that is, one diagonal from upper-left to lower-right and one from upper-right to lower-left). However, the grid $[n]^2$ contains approximately $2n$ diagonals of each type. Therefore, if one chooses a diagonal uniformly at random the probability that it is unoccupied is approximately $1/2$. If we fix $(r,c)$ and choose $(x,y) \in Q(T)$ uniformly at random, we might imagine that the (four) diagonals containing $(r,y)$ and $(x,c)$ are distributed uniformly at random, which would imply that with constant probability they are unoccupied.

One approach to proving Lemma \ref{lem:an abundance of absorbers} is to track $|\mB_{Q(t)}(r,c)|$ for each $(r,c) \in [n]^2$ and each $0 \leq t \leq T$. This would give a precise count of $|\mB_{Q(t)}(r,c)|$ but at the cost of significant technical work. We take a different tack, giving a shorter proof at the expense of a suboptimal $\beta$. We show that already at an early stage of the process w.h.p.\ there are $\Omega(n)$ absorbers for each position. Furthermore, each of these absorbers is guaranteed to be present in $Q(T)$. We begin by identifying a combinatorial property that ensures an absorber survives throughout the random greedy process.

We will need the following notation: Let $(r,c) \in [n]^2$. Let $\elltor{1}(r,c) \coloneqq \{ (x,y) \in [n]^2 : r+c = x+y \bmod n \},\elltor{2}(r,c) \coloneqq \{ (x,y) \in [n]^2 : r-c = x-y \bmod n \}$ be the two toroidal diagonals incident to $(r,c)$, and let $\ell_1(r,c) \subseteq \elltor{1}(r,c), \ell_2(r,c) \subseteq \elltor{2}(r,c)$ be the two non-toroidal diagonals incident to $(r,c)$. For $i=1,2$ we also define $\ellfar{i}(r,c) \coloneqq \elltor{i}(r,c) \setminus \ell_i(r,c)$.

\begin{definition}
	Let $Q \subseteq [n]^2$ be a partial $n$-queens configuration. Let $(r,c) \in [n]^2$. An absorber $(x,y) \in \mB_Q(r,c)$ is \termdefine{safe} if the four toroidal diagonals containing $(r,y)$ and $(x,c)$ are all occupied by $Q$. Equivalently, $(x,y) \in \mB_Q(r,c)$ is safe if there exist $a_1 \in \ellfar{1}(r,y), a_2 \in \ellfar{2}(r,y), b_1 \in \ellfar{1}(x,c), b_2 \in \ellfar{2}(x,c)$ such that $a_1,a_2,b_1,b_2 \in Q$. (see Figure \ref{fig:safe absorbers}).
	
	More generally, if $R \subseteq [n]^2$ (i.e., $R$ is not necessarily a partial queens configurations) and $(r,c) \in [n]^2$, we say that $(x,y) \in R$ is a safe absorber for $(r,c)$ if $(x,y)$ does not share a diagonal with $(r,c)$, there exist $a_1 \in \ellfar{1}(r,y), a_2 \in \ellfar{2}(r,y), b_1 \in \ellfar{1}(x,c), b_2 \in \ellfar{2}(x,c)$ such that $(x,y),a_1,a_2,b_1,b_2 \in R$ and no other element of $R$ shares a line with any of $(x,y),a_1,a_2,b_1,b_2 \in R$.
	
	We denote the set of safe absorbers for $(r,c)$ in $R$ by $\mS_R(r,c)$.
\end{definition}

	\begin{figure}
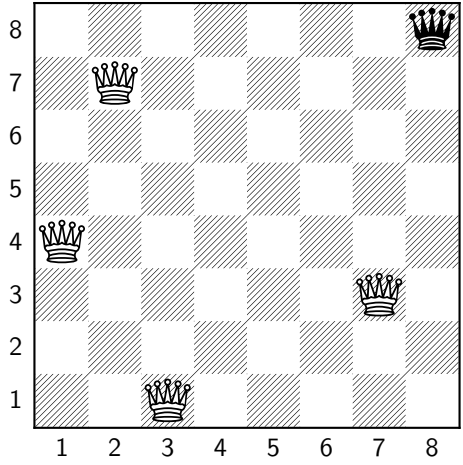

	\centering
	{   \def\whitepieces{qa4, qb7, qc1, qg3}
		\def\blackpieces{qh8}
		\chessboard[ showmover=false,labelbottomformat=\arabic{filelabel}, setwhite=\whitepieces,addblack=\blackpieces]}
	%emphstyle=\color{yellow},emphfields=\emphfields]
	
	\caption{In the above partial $8$-queens configuration, the black queen at $(8,8)$ is a safe absorber for the square $(5,4)$. The four white queens occupy the toroidal diagonals incident to $(5,8)$ and $(8,4)$, while the (non-toroidal) diagonals incident to those squares are unoccupied.}
	\label{fig:safe absorbers}
\end{figure}

The crucial property of safe absorbers is the following.

\begin{obs}\label{obs:safe absorbers survive}
	Let $Q_1 \subseteq Q_2 \subseteq [n]^2$ be partial toroidal $n$-queens configurations. Then, for every $(r,c) \in [n]^2$, it holds that $\mS_{Q_1}(r,c) \subseteq \mS_{Q_2}(r,c)$.
\end{obs}

\begin{proof}
	Since $Q_2$ is a partial toroidal $n$-queens configuration, every toroidal diagonal contains at most a single queen. Since $Q_1 \subseteq Q_2$, if a toroidal diagonal is occupied in $Q_1$ then it will contain no additional queens in $Q_2$.
	
	Let $(r,c) \in [n]^2$ and suppose $(x,y) \in \mS_{Q_1}(r,c)$. Then, by definition, the toroidal diagonals incident to $(r,y)$ and $(x,c)$ are occupied in $Q_1$, while the (non-toroidal) diagonals incident with these positions are unoccupied. Hence, this remains so in $Q_2$ as well, implying $(x,y) \in \mS_{Q_2}(r,c)$.
\end{proof}

In order to prove Lemma \ref{lem:an abundance of absorbers} we will show that for some $0 \leq t \leq T$, $Q(t)$ contains $\Omega(n)$ safe absorbers for every $(r,c) \in [n]^2$. In order to gain a heuristic understanding we consider the situation in a random binomial subset of $[n]^2$. This calculation will also be used in the proof of Lemma \ref{lem:an abundance of absorbers}.

\begin{clm}\label{clm:safe absorber probability}
	Let $R \subseteq [n]^2$ be a random binomial subset where each position is present with probability $p \in [0,1]$, independently of other positions. Let $(r,c),(x,y) \in [n]^2$ be two positions that do not share a line. Then:
	\[
	\Prob [ (x,y) \in \mS_R(r,c) ] \geq |\ellfar{1}(r,y)| |\ellfar{2}(r,y)| |\ellfar{1}(x,c)| |\ellfar{2}(x,c)| p^5 (1-p)^{20n} - O \left( \frac{1}{n^2} \right).
	\]
\end{clm}

\begin{proof}
	Instead of asking for the probability that $(x,y) \in \mS_{R}(r,c)$, we will instead estimate the probabilities of the following events.
	\begin{itemize}
		\item Let $A$ be the event that $(x,y) \in R$ and no other position sharing a line with $(x,y)$ is in $R$. Then
		\[
		\Prob [A] = p(1-p)^{4n-1} \geq p(1-p)^{4n}.
		\]
		
		\item Let $B$ be the event that there exist distinct positions $a_1 \in \ellfar{1}(r,y),a_2 \in \ellfar{2}(r,y), b_1 \in \ellfar{1}(x,c), b_2 \in \ellfar{2}(x,c)$ such that $a_1,a_2,b_1,b_4 \in R$, and no line incident to any of these positions contains more than one element of $R$. Since every pair of lines intersects in at most one position, there are at least 
		\begin{align*}
		\left(|\ellfar{1}(r,y)| - O(1)\right) & \left(|\ellfar{2}(r,y)| - O(1) \right) \left(|\ellfar{1}(x,c)| - O(1)\right) \left(|\ellfar{2}(x,c)| - O(1)\right)\\
		& = |\ellfar{1}(r,y)| |\ellfar{2}(r,y)| |\ellfar{1}(x,c)| |\ellfar{2}(x,c)| - O(n^3)
		\end{align*}
		choices of distinct $a_1 \in \ellfar{1}(r,y),a_2 \in \ellfar{2}(r,y), b_1 \in \ellfar{1}(x,c), b_2 \in \ellfar{2}(x,c)$ that do not share a line with each other or with $(x,y)$. Therefore:
		\begin{align*}
		\Prob [B] & \geq \left( |\ellfar{1}(r,y)| |\ellfar{2}(r,y)| |\ellfar{1}(x,c)| |\ellfar{2}(x,c)| - O(n^3) \right) \left( p (1-p)^{4n} \right)^4\\
		& \geq |\ellfar{1}(r,y)| |\ellfar{2}(r,y)| |\ellfar{1}(x,c)| |\ellfar{2}(x,c)| \left( p (1-p)^{4n} \right)^4 - O \left( \frac{1}{n} \right).
		\end{align*}
	\end{itemize}
	Note that if both $A$ and $B$ occur then $(x,y) \in \mS_R(r,c)$. We observe that $(x,y)$ is not contained in any of the lines $\ellfar{1}(r,y), \ellfar{2}(r,y), \ellfar{1}(x,c), \ellfar{2}(x,c)$. Therefore $\Prob[B \given A] \geq \Prob [B]$. Hence
	\begin{equation*}
	\begin{split}
	\Prob & [ (x,y) \in \mS_R(r,c) ] \geq \Prob \left[A \cap B\right] = \Prob[A]\Prob[B \given A] \geq \Prob [A] \Prob [B]\\
	& \geq p(1-p)^{4n} \left( |\ellfar{1}(r,y)| |\ellfar{2}(r,y)| |\ellfar{1}(x,c)| |\ellfar{2}(x,c)| \left( p (1-p)^{4n} \right)^4 - O \left( \frac{1}{n} \right) \right)\\
	& \geq |\ellfar{1}(r,y)| |\ellfar{2}(r,y)| |\ellfar{1}(x,c)| |\ellfar{2}(x,c)| p^5 (1-p)^{20n} - O \left( \frac{1}{n^2} \right).
	\end{split}
	\end{equation*}
\end{proof}

The preceding calculation suggests that we should consider positions $(x,y)$ such that $|\ellfar{1}(r,y)|, |\ellfar{2}(r,y)|, |\ellfar{1}(x,c)|, |\ellfar{2}(x,c)|, = \Omega(n)$. As we will apply Claim \ref{clm:safe absorber probability} with $p = \Omega(1/n)$, this condition would imply that $\Prob[ (x,y) \in \mS_R(r,c) ] = \Omega(1/n)$. This motivates the next definition.

\begin{definition}
	A position $(x,y) \in [n]^2$ is \termdefine{balanced} if $|\ellfar{1}(x,y)|,|\ellfar{2}(x,y)| \geq n/10$.
\end{definition}

We will now show that for every $(r,c) \in [n]^2$ there are $\Omega(n^2)$ positions $(x,y) \in [n]^2$ such that $(r,c)$ and $(x,y)$ do not share a line and $(r,y)$ and $(x,c)$ are both balanced. Let $S \subseteq [n]^2$ be the set of positions $(x,y)$ such that neither of
\[
x - n/10 \leq y \leq x + n/10, \quad \frac{9}{10}n - x \leq y \leq \frac{11}{10}n - x
\]
holds (see Figure \ref{fig:balanced}).

\begin{figure}
	\centering
	\includegraphics[scale=0.8]{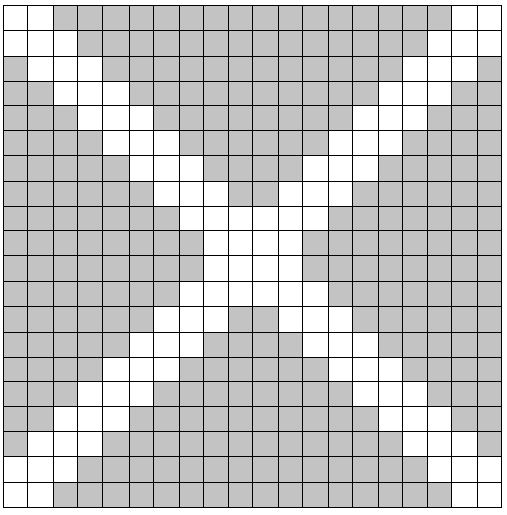}
	\caption{ Roughly speaking, \textit{balanced} means not too close to the main diagonals. In the $20 \times 20$ array above the set $S$ of balanced positions is shaded.}
	\label{fig:balanced}
\end{figure}

\begin{clm}
	Let $(x,y) \in S$. Then $(x,y)$ is balanced.
\end{clm}

\begin{proof}
	For every $(x,y) \in [n]^2$, we have $|\ellfar{1}(x,y)| = |x+y-(n+1)|$ and $|\ellfar{2}(x,y)| = |x-y|$.
	
	Let $(x,y) \in S$. Then, by definition, either $y > x+n/10$ or $x > y+n/10$. Both imply $|\ellfar{2}(x,y)| = |x-y| > n/10$. Similarly, either $x+y > 11n/10$ or $x+y < 9n/10$. In the first case, $x+y - (n+1) > n/10 - 1$. In the second, $x+y-(n+1) < -n/10-1$. Both imply $|\ellfar{1}(x,y)| = |x+y-(n+1)| \geq n/10$.
\end{proof}

\begin{clm}\label{clm:S intersection large}
	Each row and each column in $[n]^2$ intersects $S$ on at least $n/2$ positions.
\end{clm}

\begin{proof}
	Since $S$ is invariant under reflection along the $y=x$ axis it suffices to prove the claim for rows. $S$ is also invariant under the reflection $(r,c) \mapsto (n+1-r,c)$, so it suffices to prove the claim for rows $r \leq (n+1)/2$.
	
	Let $r \in [n]$ be at most $(n+1)/2$. We will show that row $r$ contains at least $n/2$ positions in $S$. Indeed, if $1 \leq r \leq n/10$ then for every $2n/10 < c < 8n/10$, it holds that $(r,c) \in S$. Otherwise, $n/10 < r \leq (n+1)/2$. In this case every $c$ satisfying one of $n/10 + r < c < 9n/10 - r$, $c < r - n/10$, or $c > 11n/10 - r$ will also satisfy $(r,c) \in S$. In any case, there are more than $n/2$ elements $c \in [n]$ such that $(r,c) \in S$.
\end{proof}

\begin{clm}\label{clm:many balanced positions}
	For every $(r,c) \in [n]^2$ there are at least $n^2/5$ positions $(x,y) \in [n]^2$ such that $(r,y)$ and $(x,c)$ are balanced and $(r,c)$ and $(x,y)$ do not share a line.
\end{clm}

\begin{proof}
	Let $(r,c) \in [n]^2$. Let $S_r \coloneqq \{y : (r,y) \in S, y \neq c \}$ and $S_c \coloneqq \{ x : (x,c) \in S, x \neq r \}$. By Claim \ref{clm:S intersection large}, $|S_r|,|S_c| \geq n/2 - 1$. Now, let
	\[
	S_{(r,c)} \coloneqq \left(S_c \times S_r\right) \setminus \left( \ell_1(r,c) \cup \ell_2(r,c) \right).
	\]
	We note that $|S_{(r,c)}| \geq (n/2 - 1)^2 - 2n \geq n^2 / 5$. Also, by definition, if $(x,y) \in S_{(r,c)}$ then it does not share a line with $(r,c)$. Next, we argue that for every $(x,y) \in S_{(r,c)}$, it holds that $(r,y)$ and $(x,c)$ are both balanced. Indeed, if $(x,y) \in S_{(r,c)}$, then (again by definition) $x \in S_c$. Therefore $(x,c) \in S$ and so by Claim \ref{clm:S intersection large} is balanced. Similarly, $(r,y) \in S$ and hence is balanced as well.
\end{proof}

In order to prove Lemma \ref{lem:an abundance of absorbers} we will use the following bounded-differences inequality.

\begin{thm}[\cite{warnke2016method} Corollary 1.4]\label{thm:bounded differences}
	Let $X_1,\ldots,X_N$ be i.i.d.\ Bernoulli random variables with $\Prob[X_1 = 1] = p$. Assume that for $K>0$ the function $f:\{0,1\}^N \to \R$ satisfies the Lipschitz condition $\left| f(\omega) - f(\omega') \right| \leq K$ whenever $\omega,\omega' \in \{0,1\}^N$ differ by a single coordinate. Then, for all $t > 0$:
	\[
	\Prob \left[ \left|f(X_1,\ldots,X_N) - \E \left[ f(X_1,\ldots,X_N) \right]\right| \geq t \right] \leq 2 \exp \left( - \frac{t^2}{2N(1-p)pK^2 + 2Kt/3} \right).
	\]
\end{thm}

Theorem \ref{thm:bounded differences} applies to functions of Bernoulli random variables. In order to apply it to our situation, we introduce a coupling of $\{Q(t)\}_{t=0}^T$ to a binomial random subset of $[n]^2$. For every $x \in [n]^2$, let $x_\alpha \in [0,1]$ be i.i.d.\ uniformly random. Observe that for every $q \in [0,1]$, $R(q) \coloneqq \{ x \in [n]^2 : x_\alpha < q \}$ is a binomial random subset of $[n]^2$ with density $q$.

Now consider the random process $\tilde{Q}(0),\tilde{Q}(1),\ldots$, defined as follows: $\tQ(0) = \emptyset$. Suppose that for $t<n$ we have defined $\tQ(t)$. Let $x \in [n]^2$ be the position that minimizes $x_\alpha$ which is also available in $\tilde{Q}(t)$ (by ``available'' we mean not sharing a row, column, or toroidal diagonal with a previously chosen position). Set $\tilde{Q}(t+1) = \tilde{Q}(t) \cup \{x\}$. If there is no such position, set $\tQ(t+1)=\tQ(t)$. Observe that $\{ \tilde{Q}(t) \}_{t=0}^T$ and $\{Q(t)\}_{t=0}^T$ have the same distribution. Thus we may (and do) identify them. We remark, however, that the process $\tQ(t)$ is defined also for $T < t \leq n$. This will be useful in the proof of the next claim.

Set $p = 1/(4n)$ and let $R = R(p)$. Let $\tR \subseteq R$ be the set of positions $x \in R$ that are not threatened by any other element of $R$ (that is, they are the only positions in their row, column, and toroidal diagonals).

\begin{clm}\label{clm:tR contained in QT}
	With probability $1 - \exp \left( -\Omega(n) \right)$ it holds that $\tR \subseteq Q(T)$.
\end{clm}

\begin{proof}
	We note that by Chernoff's inequality we have $\Prob \left[ |R| > T \right] = \exp \left( - \Omega(n) \right)$.
	
	Let $x \in \tR$. By definition of the process, it holds that $x_\alpha < y_\alpha$ for every $y$ that threatens $x$. Thus, there exists some $t$ (possibly greater than $T$) such that $x \in \tQ(t)$. Let $t_x$ be the smallest $t$ such that $x \in \tQ(t)$. It suffices to show that $t_x \leq T$. This clearly holds if $|R| \leq T$. The latter event holds with probability $1 - \exp \left( -\Omega(n) \right)$, as desired.
\end{proof}

Due to Observation \ref{obs:safe absorbers survive} and Claim \ref{clm:tR contained in QT}, with probability $1 - \exp \left( -\Omega(n) \right)$ we have $\mS_{\tR}(r,c) \subseteq \mS_{Q(T)}(r,c)$ for every $(r,c) \in [n]^2$. Therefore, Lemma \ref{lem:an abundance of absorbers} follows from the following claim.

\begin{clm}\label{clm:tR has lots of absorbers}
	There exists a constant $\beta > 0$ such that w.h.p.\ $\tR$ is $(\beta n)$-absorbing.
\end{clm}

We will apply Theorem \ref{thm:bounded differences} to prove Claim \ref{clm:tR has lots of absorbers}. We will use the following Lipschitz condition.

\begin{clm}\label{clm:tR Lipschitz}
	Let $Q_1,Q_2 \subseteq [n]^2$ be partial toroidal $n$-queens configurations such that $|Q_1 \Delta Q_2| = 1$. Then, for every $(r,c) \in [n]^2$, we have $||\mS_{Q_1}(r,c)| - |\mS_{Q_2}(r,c)|| \leq 5$.
\end{clm}

\begin{proof}
	Let $(r,c) \in [n]^2$. It suffices to show that if $R_2$ is obtained from $R_1$ by adding a queen $(x,y)$, then $||\mS_{Q_1}(r,c)| - |\mS_{Q_2}(r,c)|| \leq 5$.
	
	Since $Q_1 \subseteq Q_2$, Observation \ref{obs:safe absorbers survive} implies that $\mS_{Q_1}(r,c) \subseteq \mS_{Q_2}(r,c)$. Hence, it is enough to show that
	\[
	| \mS_{Q_2}(r,c) \setminus \mS_{Q_1}(r,c) | \leq 5.
	\]
	
	We observe that overall, $R_2$ contains at most one more (not necessarily safe) absorber for $(r,c)$ than $R_1$. Indeed, the only candidate for a new absorber in $R_2$ is $(x,y)$ itself. In this way, $\mS_{R_2}(r,c)$ may contain at most one absorber that is not in $\mB_{R_1}(r,c)$. Any other elements of $\mS_{Q_2}(r,c) \setminus \mS_{Q_1}(r,c)$ must have the following form: There exists some $(x',y') \in \mB_{Q_1}(r,c) \setminus \mS_{Q_1}(r,c)$ and $(x',y') \in \mS_{Q_2}(r,c)$. This occurs only if $(x,y) \in \ellfar{i}(x',y')$ for some $i=1,2$ and some $(x',y') \in \mB_{Q_1}(r,c)$. In this case $(x,y)$ shares a toroidal diagonal with one of $(r,y')$ and $(x',c)$. Hence, one of the following hold modulo $n$:
	\[
	x+y = r+y', \quad x-y = r-y', \quad x+y = x'+c, \quad x-y = x'-c.
	\]
	For each $x'$ that solves one of these equations, there is at most one $y'$ such that $(x',y') \in Q_1$. Similarly, for each $y'$ solving an equation there is at most one $x'$ such that $(x',y') \in Q_1$. Therefore, in this way, adding $(x,y)$ to $Q_1$ may contribute at most $4$ absorbers to $\mS_{Q_2}(r,c)$. In summary:
	\[
	| \mS_{Q_2}(r,c) \setminus \mS_{Q_1}(r,c) | \leq 5,
	\]
	as desired.
\end{proof}

\begin{proof}[Proof of Claim \ref{clm:tR has lots of absorbers}]
	We will now show that for $\beta \coloneqq (40e)^{-5}$ with high probability $\tR$ contains at least $\beta n$ safe absorbers for every position.
	
	Let $(r,c) \in [n]^2$. We will show that
	\[
	\Prob \left[ |\mS_{\tR}(r,c)| < \beta n \right] \leq \exp \left( - \Omega (n) \right).
	\]
	A union bound over the $n^2$ positions will then imply that w.h.p.\ this holds simultaneously for all positions.
	
	Let $X_{(r,c)} \coloneqq |\mS_{\tR}(r,c)|$. We will apply Theorem \ref{thm:bounded differences} to $X_{(r,c)}$. For $(x,y) \in [n]^2$, let $Y_{(x,y)}$ indicate the event $(x,y) \in R$. Then the random variables $\{Y_{(x,y)}\}_{(x,y) \in [n]^2}$ are independent. Furthermore, each $Y_{(x,y)}$ has Bernoulli distribution with parameter $p$. Additionally, $R, \tR$, and hence $X_{(r,c)}$ are determined by these random variables. Therefore $X_{(r,c)}$ is a function of the i.i.d.\ Bernoulli random variables $\{Y_{(x,y)}\}_{(x,y) \in [n]^2}$.
	
	We note that adding a queen $(x,y)$ to $R$ changes $\tR$ by at most $4$. Indeed, if $(x,y)$ is the unique queen in the lines incident to it in $R$, then $|\tR|$ increases by one. On the other hand, $(x,y)$ occupies four lines. If any of these contains a queen from $\tR$, then that queen will be removed from $\tR$ once $(x,y)$ is added to $R$. Hence, $|\tR|$ may decrease by at most $4$. Hence, By Claim \ref{clm:tR Lipschitz}, changing the value of some $Y_{(x,y)}$ can change the value of $X_{(r,c)}$ by at most $20$.
	
	Next we bound $\E X_{(r,c)}$ from below, with the goal of showing that $\E X_{(r,c)} - \beta n = \Omega(n)$. By Claim \ref{clm:many balanced positions} there are at least $n^2/5$ positions $(x,y)$ such that $(r,c)$ and $(x,y)$ do not share a line and $(r,y),(x,c)$ are balanced. Let $P$ be the set of these positions. For each $(x,y) \in P$ let $Z_{(x,y)}$ indicate the event that $(x,y) \in \mS_{\tR}(r,c)$. Then
	\[
	X_{(r,c)} \geq \sum_{(x,y) \in P} Z_{(x,y)}.
	\]
	By Claim \ref{clm:safe absorber probability}:
	\[
	\E Z_{(x,y)} = \Prob [Z_{(x,y)} = 1] \geq |\ellfar{1}(r,y)| |\ellfar{2}(r,y)| |\ellfar{1}(x,c)| |\ellfar{2}(x,c)| p^5 (1-p)^{20n} - O \left( \frac{1}{n^2} \right).
	\]
	If $(x,y) \in P$ then $(r,y)$ and $(x,c)$ are balanced. Thus:
	\[
	|\ellfar{1}(r,y)| |\ellfar{2}(r,y)| |\ellfar{1}(x,c)| |\ellfar{2}(x,c)| \geq \left( \frac{n}{10} \right)^4.
	\]
	Hence
	\[
	\E Z_{(x,y)} \geq \left( \frac{n}{10} \right)^4 p^5 (1-p)^{20n} - O \left( \frac{1}{n^2} \right) \stackrel{p = 1/(4n)}{=} \frac{1}{(4e)^5 10^4n} - O \left( \frac{1}{n^2} \right).
	\]
	Therefore:
	\[
	\E X_{(r,c)} \geq \sum_{(x,y) \in P} \E Z_{(x,y)} \geq \frac{|P|}{(4e)^5 10^4n} - O \left( \frac{|P|}{n^2} \right) \geq \frac{n}{5 (4e)^5 10^4} - O(1) = 2\beta n - O(1).
	\]
	Let $\lambda \coloneqq \E X_{(r,c)} - \beta n$. Then $\lambda \geq \beta n - O(1) = \Omega (n)$. By Theorem \ref{thm:bounded differences}:
	\begin{align*}
	\Prob \left[ X_{(r,c)} < \beta n \right] = \Prob \left[ X_{(r,c)} < \E X_{(r,c)} - \lambda \right] & \leq 2 \exp \left( - \frac{\lambda^2}{2n^2(1-p)p 20^2 + 40 \lambda / 3} \right)\\
	& \leq \exp \left( - \Omega (n) \right).
	\end{align*}
	This completes the proof of Claim \ref{clm:tR has lots of absorbers}, and hence also of Lemma \ref{lem:an abundance of absorbers}.
\end{proof}

\section{Proof of Theorem \ref{thm:main}}

In this section we use the results of Sections \ref{sec:rga} and \ref{sec:absorbers} to prove the enumeration in Theorem \ref{thm:main}. The idea is to bound, from below, the number of successful outcomes of the random process. Then, we divide this by (an upper bound on) the number of ways in which each particular $n$-queens configuration can be constructed by the process. The result is a lower bound on $\mQ(n)$.

We begin by counting the number of outcomes in the random greedy phase (Algorithm \ref{alg:rga}). We have shown that w.h.p.\ this algorithm succeeds in constructing a partial $n$-queens configuration with $T = \lfloor (1-n^{-\alpha})n \rfloor$ queens. Furthermore, w.h.p.\ for every $0 \leq t < T$, the number of available positions at time $t$ satisfies $|\mA(t)| = \left( 1 \pm n^{-0.4} \right) n^2 (1-t/n)^4$. Thus, the number of outcomes of the random greedy phase is at least
\[
X \coloneqq \prod_{t=0}^{T-1} \left( 1 \pm n^{-0.4} \right) n^2 (1-t/n)^4 \geq \left( \left( 1 - n^{-\alpha} \right)n^2\right)^T \exp \left( 4 \sum_{t=0}^{T-1} \log \left( 1 - t/n \right) \right).
\]
Turning our attention to the sum, we have:
\[
\sum_{t=0}^{T-1} \log \left( 1 - t/n \right) = n \frac{1}{n} \sum_{t=0}^{T-1} \log \left( 1 - t/n \right)
\geq n \int_0^{T/n} \log(1-x)dx \geq - n \left( 1 + \tO(n^{-\alpha}) \right).
\]
Hence:
\[
X \geq \left( \left( 1 - \tO(n^{-\alpha}) \right) n^2\right)^T \exp \left( -4n - \tO(n^{1-\alpha}) \right) \geq \left( \left(1 - \tO(n^{-\alpha}) \right) \frac{n^2}{e^4}\right)^n.
\]
Furthermore, by Lemma \ref{lem:an abundance of absorbers} all but a vanishing fraction of these outcomes are $10 (n-T)$-absorbing, and therefore can be completed to a complete configuration by Algorithm \ref{alg:absorption}.

We now wish to bound, from above, the number of ways in which a specific $n$-queens configuration can be constructed by the algorithm. Given such a configuration, $2(n-T)$ of its queens were placed by the absorbing stage. Given $2(n-T)$ queens, there are $(2(n-T)-1)!! = (2(n-T))! / (2^{n-T}(n-T)!)$ ways to partition them into $n-T$ pairs. For each pair, there are (at most) $2$ ways it might have been obtained as an absorber. Finally, there are now $T!$ orders in which the remaining queens might have been placed on the board. Thus, the number of ways to obtain the configuration is at most
\begin{align*}
Y \coloneqq \binom{n}{2(n-T)} (2(n-T))!! \times 2^{n-T} \times T!
& \leq n^{2(n-T)} (2(n-T))^{2(n-T)} 2^{n-T} n!\\
& \leq \left(  1 + \tO(n^{-\alpha}) \right)^n n!.
\end{align*}
Applying Stirling's approximation:
\[
Y \leq \left(\left(1 + \tO(n^{-\alpha})\right) \frac{n}{e} \right)^n.
\]
Therefore, the number of $n$-queens configurations satisfies
\[
\mQ(n) \geq \frac{X}{Y} \geq \frac{\left( \left(1 - \tO(n^{-\alpha}) \right) \frac{n^2}{e^4}\right)^n}{\left(\left(1 + \tO(n^{-\alpha})\right) \frac{n}{e} \right)^n} \geq
\left( \left(1 - \tO(n^{-\alpha}) \right) \frac{n}{e^3}\right)^n,
\]
as desired.

\section{Concluding remarks}

\begin{itemize}
	\item We have established that $\liminf_{n\to\infty} \frac{\mQ(n)^{1/n}}{n} \geq e^{-3}$. In \cite{luria2017new} the first author proved that for $\alpha \approx 1.587$, $\limsup_{n\to\infty} \frac{\mQ(n)^{1/n}}{n} \leq e^{-\alpha}$. It would be interesting to close this gap. In particular, \cite{zhang2009counting} provides numerical evidence that for $\beta \approx 1.944$, $\lim_{n\to\infty} \frac{\mQ(n)^{1/n}}{n} = e^{-\beta}$.

	\item A curious feature of our algorithm is that although, ultimately, it constructs a \textit{non-toroidal} configuration, the random greedy algorithm actually constructs an approximate \textit{toroidal} configuration.
	
	As mentioned in the introduction the toroidal algorithm has a regularity which makes its analysis more amenable than its non-toroidal counterpart. Nevertheless, in computer simulations we have found the non-toroidal algorithm extremely successful at constructing approximate configurations. Furthermore, (for $n \geq 100$ and as large as $10^7$) we have never seen Algorithm \ref{alg:absorption} fail to obtain a complete configuration.
	
	Ours is not the first random greedy algorithm used to construct $n$-queens configurations: In \cite{minton1990solving} configurations were generated by proceeding through the board's rows from top to bottom, where in each row a queen was placed uniformly at random among the available positions. Finally, a min-conflicts algorithm was used to modify the board and find a complete solution. This algorithm was found to construct complete configurations with ease, but it was not analyzed formally. We have found this algorithm to possess a non-uniformity making its analysis challenging as well.

	Conversely, we wonder if the output of Algorithm \ref{alg:rga} can be completed to a toroidal configuration (assuming $n=1,5 \bmod 6$, which is a necessary condition). If so, this would give a lower bound of $((1-\oone)ne^{-3})^n$ on such configurations, matching an upper bound of the first author \cite[Theorem 1.2]{luria2017new}. The difficulty we have encountered is that in an approximate toroidal configuration nearly all toroidal diagonals are occupied. This makes absorbers like those in our proof difficult to find.
	
	\item The list of conjectures at the end of \cite{bell2009survey} contains several generalizations of the $n$-queens problem, for example to boards with different topologies and higher dimensions. It would be interesting if these problems could be attacked with techniques similar to the one in this paper.
\end{itemize}

\bibliography{torus}

\providecommand{\bysame}{\leavevmode\hbox to3em{\hrulefill}\thinspace}
\providecommand{\MR}{\relax\ifhmode\unskip\space\fi MR }
\providecommand{\MRhref}[2]{%
  \href{http://www.ams.org/mathscinet-getitem?mr=#1}{#2}
}
\providecommand{\href}[2]{#2}
\begin{thebibliography}{10}

\bibitem{bell2009survey}
Jordan Bell and Brett Stevens, \emph{A survey of known results and research
  areas for n-queens}, Discrete Mathematics \textbf{309} (2009), no.~1, 1--31.

\bibitem{bohman2009triangle}
Tom Bohman, \emph{The triangle-free process}, Advances in Mathematics
  \textbf{221} (2009), no.~5, 1653--1677.

\bibitem{bohman2019large}
Tom Bohman and Lutz Warnke, \emph{Large girth approximate {Steiner} triple
  systems}, Journal of the London Mathematical Society \textbf{100} (2019),
  no.~3, 895--913.

\bibitem{glock2016existence}
Stefan Glock, Daniela K{\"u}hn, Allan Lo, and Deryk Osthus, \emph{The existence
  of designs via iterative absorption}, arXiv preprint arXiv:1611.06827 (2016).

\bibitem{glock2020conjecture}
\bysame, \emph{On a conjecture of {Erd{\H{o}}s} on locally sparse {Steiner}
  triple systems}, Combinatorica \textbf{40} (2020), no.~3, 363--403.

\bibitem{keevash2014existence}
Peter Keevash, \emph{The existence of designs}, arXiv preprint arXiv:1401.3665
  (2014).

\bibitem{keevash2018counting}
\bysame, \emph{Counting designs}, Journal of the European Mathematical Society
  \textbf{20} (2018), no.~4, 903--927.

\bibitem{luria2017new}
Zur Luria, \emph{New bounds on the number of $n$-queens configurations}, arXiv
  preprint arXiv:1705.05225 (2017).

\bibitem{minton1990solving}
Steven Minton, Mark~D Johnston, Andrew~B Philips, and Philip Laird,
  \emph{Solving large-scale constraint-satisfaction and scheduling problems
  using a heuristic repair method.}, AAAI, vol.~90, 1990, pp.~17--24.

\bibitem{pauls1874maximalproblem}
E~Pauls, \emph{Das {Maximalproblem} der {Damen} auf dem {Schachbrete}},
  Deutsche Schachzeitung \textbf{29} (1874), 261--263.

\bibitem{pauls1874maximalproblemII}
\bysame, \emph{Das {Maximalproblem} der {Damen} auf dem {Schachbrete}, {II}},
  Deutsche Schachzeitung \textbf{29} (1874), 257--267.

\bibitem{rivin1994n}
Igor Rivin, Ilan Vardi, and Paul Zimmermann, \emph{The n-queens problem}, The
  American Mathematical Monthly \textbf{101} (1994), no.~7, 629--639.

\bibitem{rucinski1992random}
Andrzej Ruci{\'n}ski and Nicholas~C Wormald, \emph{Random graph processes with
  degree restrictions}, Combinatorics, Probability and Computing \textbf{1}
  (1992), no.~2, 169--180.

\bibitem{warnke2016method}
Lutz Warnke, \emph{On the method of typical bounded differences},
  Combinatorics, Probability and Computing \textbf{25} (2016), no.~2, 269--299.

\bibitem{wormald1995differential}
Nicholas~C Wormald, \emph{Differential equations for random processes and
  random graphs}, The annals of applied probability \textbf{5} (1995), no.~4,
  1217--1235.

\bibitem{zhang2009counting}
Cheng Zhang and Jianpeng Ma, \emph{Counting solutions for the {$N$}-queens and
  {Latin}-square problems by {Monte} {Carlo} simulations}, Physical Review E
  \textbf{79} (2009), no.~1, 016703.

\end{thebibliography}
\bibliographystyle{amsplain}

\end{document}